\documentclass{article}
\usepackage{amssymb}
\usepackage{amsthm}
\usepackage{amsmath}

{\theoremstyle{plain}%
\newtheorem{theorem}{Theorem}[section]
\newtheorem{corollary}{Corollary}[section]
\newtheorem{proposition}{Proposition}[section]
\newtheorem{lemma}{Lemma}[section]
\newtheorem{definition}{Definition}[section]
\newtheorem{remark}{Remark}[section]
\newtheorem{notation}{Notation}[section]
\newtheorem{example}{Example}[section]

\newfont{\hueca}{msbm10}
\def\hu #1{\hbox{\hueca #1}}\def\hu #1{\hbox{\hueca #1}}
\begin{document}

\title{On split Leibniz superalgebras}
\author{Antonio J. Calder\'{o}n Mart\'{\i }n
\thanks{Supported by the PCI of the
UCA `Teor\'\i a de Lie y Teor\'\i a de Espacios de Banach', by the
PAI with project numbers
 FQM298, FQM7156 and by the project of the Spanish Ministerio de Educaci\'on y Ciencia
 MTM2010-15223.}\\
Departamento de Matem\'{a}ticas. \\
Universidad de C\'{a}diz. 11510 Puerto Real, C\'{a}diz, Spain.\\
e-mail: ajesus.calderon@uca.es\\
\\
Jos\'{e} M.  S\'{a}nchez-Delgado \\
Departamento \'{A}lgebra, Geometr\'{i}a y Topolog\'{i}a.\\
Universidad de M\'alaga. 29071 M\'{a}laga, Spain.\\
e-mail: txema.sanchez@uma.es}
\date{}
\maketitle

\begin{abstract}
We study the structure of arbitrary split Leibniz superalgebras.
We show that any of such superalgebras ${\frak L}$ is of the form
${\frak L} = {\mathcal U} + \sum_jI_j$ with ${\mathcal U}$ a
subspace of an abelian (graded) subalgebra $H$ and any $I_j$ a
well described (graded) ideal of ${\frak L}$ satisfying $[I_j,I_k]
= 0$ if $j \neq k$. In the case of ${\frak L}$ being of maximal
length, the simplicity of ${\frak L}$ is also characterized in
terms of connections of roots.

\medskip


\end{abstract}

\section{Introduction and previous definitions}



Throughout this paper, Leibniz superalgebras ${\frak L}$ are
considered of arbitrary dimension and over an arbitrary base field
${\hu K}$. It is worth to mention that, unless otherwise stated,
there is not any restriction on $\dim {\frak L}_{\alpha}$, the
products $[{\frak L}_{\alpha},{\frak L}_{-\alpha}]$, or $\{k \in
\hu{K}: k\alpha \in \Lambda\}$, where ${\frak L}_{\alpha}$ denotes
the root space associated to the root $\alpha$, and $\Lambda$ the
set of nonzero roots of ${\frak L}$. The class of Leibniz
superalgebras appears as an extension of the one of Leibniz
algebras (see \cite{1, 2, 3, 4, 7, 6, 8, 5}), in a similar way
that Lie superalgebras extends to Lie algebras, motivated in part
for its applications in physics. In fact,  the class of Leibniz
superalgebras also extends the one of Lie superalgebras by
removing the skew-supersymmetry, which is of interest in the
formalism of mechanics of Nambu \cite{Dale}.

In $\S2$ we develop techniques of connections of roots in the
framework of split Leibniz superalgebras so as to show that
${\frak L}$ is of the form  ${\frak L} = {\mathcal U} + \sum_jI_j$
with ${\mathcal U}$ a subspace of an  abelian (graded) subalgebra
$H$ and any $I_j$ a well described (graded) ideal of ${\frak L}$
satisfying $[I_j,I_k] = 0$ if $j \neq k$. In $\S 3$ we focuss on
those of maximal length and characterize the simplicity of this
class of superalgebras in terms of connections in the set of their
nonzero roots.

\medskip

\begin{definition}\rm
A {\it Leibniz superalgebra} ${\frak L}$ is a ${\hu Z}_2$-graded
algebra ${\frak L} = {\frak L}_{\bar 0} \oplus {\frak L}_{\bar 1}$
over an arbitrary base field ${\mathbb K}$, endowed with a
bilinear product $[\cdot, \cdot]$, whose homogenous elements $x
\in {\frak L}_{\bar i}, y \in {\frak L}_{\bar j}, {\bar i}, {\bar
j} \in {\hu Z}_2$, satisfy
$$[x,y] \subset {\frak L}_{{\bar i}+{\bar j}}$$
$$[x,[y,z]] = [[x,y],z] - (-1)^{{\bar j}{\bar k}}[[x,z],y] \hbox{
(Super Leibniz identity)}$$ for any homogenous element $z \in
{\frak L}_{\bar k}, {\bar k} \in {\hu Z}_2$.
\end{definition}

\medskip

\begin{remark}\rm
Note that Super Leibniz identity is considered by the {\it right
side} in the sense that the multiplication operators on the right
by elements in ${\frak L}_{\bar 0}$ are derivations on the
homogeneous elements. In this way is fixed in the references
\cite{1,2,3,4,5}. However, we could have considered a Super
Leibniz identity in which the multiplication operators on the left
by elements in ${\frak L}_{\bar 0}$ would act as derivations on
the homogeneous elements, as it is the case in the references
\cite{7,6,8}. Of course, the development of the present  job would
have been similar in this case.
\end{remark}

Clearly ${\frak L}_{\bar 0}$ is a Leibniz algebra. Moreover, if
the identity $[x,y]=-(-1)^{{\bar i}{\bar j}}[y,x]$ holds we have
that Super Leibniz identity becomes Super Jacobi identity and so
Leibniz superalgebras generalize both Leibniz algebras and Lie
superalgebras.

We also recall that the {\it center} of a Leibniz superalgebra
${\frak L}$ is the set
$${\mathcal Z}({\frak L}) := \{x \in {\frak L}:[x, {\frak L}]+[{\frak
L},x]=0\}.$$

The usual regularity concepts will be understood in the graded
sense. That is, a  {\it supersubalgebra} $A$ of ${\frak L}$ is a
graded subspace $A = A_{\bar 0} \oplus A_{\bar 1}$ satisfying
$[A,A] \subset A$; and an {\it ideal} $I$ of ${\frak L}$ is a
graded subspace $I = I_{\bar 0} \oplus I_{\bar 1}$ of ${\frak L}$
such that $$[I,{\frak L}] + [{\frak L},I] \subset I.$$

The (graded) ideal ${\frak I}$ generated by $$\{[x,y] +
(-1)^{\bar{i} \bar{j}}[y,x] : x \in {\frak L}_{\bar{i}}, y \in
{\frak L}_{\bar{j}}, \bar{i}, \bar{j} \in {\hu Z}_2\}$$ plays an
important role in the theory since it determines the (possible)
non-super Lie character of ${\frak L}$. From Super Leibniz
identity, it is straightforward to check that this ideal satisfies
\begin{equation}\label{equi}
[{\frak L},{\frak I}] = 0.
\end{equation}

Here we note that the usual definition of simple superalgebra lack
of interest in the case of Leibniz superalgebras because would
imply the ideal ${\frak I}={\frak L}$ or ${\frak I}={0}$, being so
${\frak L}$ an abelian (product zero)  or a Lie superalgebra
respectively (see Equation (\ref{equi})). Abdykassymova and
Dzhumadil'daev introduced in \cite{Abdy, Dzhu} an  adequate
definition in the case of Leibniz algebras $(L,[\cdot, \cdot])$
by calling simple to the ones satisfying that its only ideals are
$\{0\}$, $L$ and the one generated by the set $\{[x,x]: x \in
L\}$.  Following this vain, we consider the next definition.

\begin{definition}\label{Defsimple}\rm
A Leibniz  superalgebra ${\frak L}$ will be called {\it simple} if
$[{\frak L},{\frak L}] \neq 0$ and its only (graded) ideals are
$\{0\}, {\frak I}$ and ${\frak L}$.
\end{definition}

\medskip

Let us introduce the class of split algebras in the framework of
Leibniz superalgebras in a similar way to the cases of Lie
algebras, Leibniz algebras and Lie superalgebras, (see \cite{Yo1},
\cite{Yo3} and \cite{Yo2} respectively), among other classes of
algebras.

We begin by considering a maximal abelian graded subalgebra $H =
H_{\bar 0} \oplus H_{\bar 1}$ among the abelian graded subalgebras
of ${\frak L}$.

\begin{definition}\label{def2}\rm
Denote by $H = H_{\bar 0} \oplus H_{\bar 1}$ a maximal abelian
(graded) subalgebra  of a Leibniz superalgebra ${\frak L}$. For a
linear functional $$\alpha : H_{\bar 0} \longrightarrow \hu{K},$$
we define the {\it root space} of ${\frak L}$ associated to
$\alpha$ as the subspace
$${\frak L}_{\alpha} := \{v_{\alpha} \in {\frak L} : [v_{\alpha}, h_{\bar 0}] =
\alpha(h_{\bar 0})v_{\alpha} \hspace{0.2cm} {\it for}
\hspace{0.2cm} {\it any} \hspace{0.2cm} h_{\bar 0} \in H_{\bar
0}\}.$$ The elements $\alpha \in (H_{\bar 0})^*$ satisfying
${\frak L}_{\alpha} \neq 0$ are called {\it roots} of ${\frak L}$
respect to $H$ and we denote $\Lambda := \{\alpha \in (H_{\bar
0})^* \setminus \{0\}: {\frak L}_{\alpha} \neq 0\}$. We say that
${\frak L}$ is a {\it split Leibniz superalgebra} if
$${\frak L} = H \oplus (\bigoplus\limits_{\alpha \in \Lambda}{\frak L}_{\alpha}).$$
We also say that $\Lambda$ is the {\it root system} of ${\frak
L}$.
\end{definition}

Split Lie algebras, split Leibniz algebras and split Lie
superalgebras are examples of split Leibniz superalgebras. Hence,
the present paper extends the results in \cite{Yo1}, \cite{Yo2}
and \cite{Yo3}. We also have to note that the results in the above
references are given for split Lie algebras, split Leibniz
algebras and split Lie superalgebras with a symmetric root system,
that is, satisfying that if $\alpha$ is a nonzero root then
$-\alpha$  it is also. However, our results   are developed for
split Leibniz superalgebras with a non-necessarily symmetric root
system. Hence, the results in the present paper not only extend
the previous ones for the classes of split Lie algebras, split
Leibniz algebras and split Lie superalgebras to a more general
class of algebras, (the one of split Leibniz superalgebras), but
also extend the results on these classes of algebras given in
\cite{Yo1,Yo2,Yo3} to split Lie algebras, split Leibniz algebras
and split Lie superalgebras with a non-necessarily symmetric root
system.

We would finally like  to mention the works of Albeverio, Ayupov,
Gago, Ladra, Liu, Omirov and Turdibaev in the finite-dimensional
setup, where Cartan subalgebras, roots and root spaces of Leibniz
and $n$-Leibniz algebras are studied, \cite{Omirov1,  Omirov3,
Omirov2, Omirov1.5}; and also to the papers \cite{laa1, laa2} in
the framework of Leibniz algebras.

\begin{example}\label{example1}\rm
Consider the 5-dimensional ${\mathbb Z}_2$-graded vector space
${\frak L}={\frak L}_{\bar 0} \oplus {\frak L}_{\bar 1}$, over a
base field ${\mathbb K}$ of characteristic different from 2, with
basis $\{u_1,u_2,u_3\}$ of ${\frak L}_{\bar 0}$ and  $\{e_1,e_2\}$
of ${\frak L}_{\bar 1};$   where the products on these elements
are given  by:
$$[u_2,u_1]=-u_3, \hspace{0.3cm} [u_1,u_2]=u_3, \hspace{0.3cm}
[u_1,u_3]=-2u_1,$$ $$[u_3,u_1]=2u_1,\hspace{0.3cm}
[u_3,u_2]=-2u_2, \hspace{0.3cm}[u_2,u_3]=2u_2,$$
$$[e_1,u_2]=e_2,\hspace{0.3cm} [e_1,u_3]=-e_1,
\hspace{0.3cm}[e_2,u_1]=e_1,\hspace{0.3cm}  [e_2,u_3]=e_2,$$ and
where the omitted products are equal to zero. Then ${\frak
L}={\frak L}_{\bar 0} \oplus {\frak L}_{\bar 1}$ becomes a
(non-Lie) split Leibniz superalgebra
$${\frak L}=H \oplus {\frak
L}_{\alpha} \oplus {\frak L}_{-\alpha}\oplus {\frak
L}_{\beta}\oplus {\frak L}_{-\beta} $$ where $H=\langle u_3
\rangle, $  ${\frak L}_{\alpha}=\langle u_2
 \rangle,$  ${\frak L}_{-\alpha}=\langle u_1
 \rangle,$  ${\frak L}_{\beta}=\langle e_1
 \rangle$ and ${\frak L}_{-\beta}=\langle e_2
 \rangle,$ being $\alpha, \beta: H \to {\mathbb K}$ defined by  $\alpha(\lambda u_3)= 2\lambda$
 and $\beta(\lambda u_3)= -\lambda$, $\lambda \in {\mathbb K}.$
\end{example}

\begin{example}\label{example2}\rm
Consider the $(n+4)$-dimensional complex ${\mathbb Z}_2$-graded
vector space ${\frak L}={\frak L}_{\bar 0} \oplus {\frak L}_{\bar
1}$  with basis $\{h,u,v\}$ of ${\frak L}_{\bar 0}$ and
$\{e_0,e_1,..., e_n\}$ of ${\frak L}_{\bar 1};$   with the
following table of multiplication:
$$[u,h]=2u, \hspace{0.3cm} [h,u]=-2u, \hspace{0.3cm}
[v,h]=-2v,$$ $$[h,v]=2v,\hspace{0.3cm} [u,v]=h,
\hspace{0.3cm}[v,u]=-h,$$
$$\hbox{$[e_k,h]=(n-2k)e_{k},$ for $0\leq k \leq n;$}$$
$$\hbox{$[e_k,v]=e_{k+1},$ for $0\leq k \leq n-1;$}$$
$$\hbox{$[e_k,u]=k(k-n-1)e_{k-1},$ for $1\leq k \leq n;$}$$
and  where the omitted products are equal to zero. Then ${\frak
L}={\frak L}_{\bar 0} \oplus {\frak L}_{\bar 1}$ is a (non-Lie)
split Leibniz superalgebra
$${\frak L}=H \oplus {\frak L}_{\alpha}
\oplus {\frak L}_{-\alpha}\oplus (\bigoplus\limits_{k=0}^{n}{\frak
L}_{\beta_k})$$
 where $H=\langle h \rangle, $  ${\frak L}_{\alpha}=\langle u
 \rangle,$  ${\frak L}_{-\alpha}=\langle v
 \rangle$   and ${\frak L}_{\beta_k}=\langle e_k
 \rangle,$  $k=0,...,n$; being
 $\alpha, \beta_k: H \to {\mathbb K}$ defined by  $\alpha(\lambda h)= 2\lambda$
 and $\beta_k(\lambda h)= (n-2k)\lambda$ for $k=0,...,n$ and $\lambda \in {\mathbb C}.$
\end{example}

\medskip

By returning to a general  split Leibniz superalgebra ${\frak L}$,
(Definition \ref{def2}), it is clear that the root space
associated to the zero root satisfies $H \subset {\frak L}_{0}.$
Conversely, given any $v_0 \in {\frak L}_0$ we can write $v_0 = h
+ \sum\limits_{i=1}^n v_{\alpha_i}$ with $h \in H$ and
$v_{\alpha_i} \in {\frak L}_{\alpha_i}$ for $i \in \{1,...,n\}$,
being $\alpha_i \in \Lambda$ with $\alpha_i \neq \alpha_j$ if $i
\neq j$. Hence $0 = [h + \sum\limits_{i=1}^n v_{\alpha_i},h_{\bar
0}] = \sum\limits_{i=1}^n \alpha_i(h_{\bar 0})v_{\alpha_i}$ for
any $h_{\bar 0} \in H_{\bar 0}$. So, taking into account the
direct character of the sum and that $\alpha_i \neq 0$, we have
that any $v_{\alpha_i} = 0$ and then $v_0 \in H$. Consequently $$H
= {\frak L}_0.$$

By the grading of ${\frak L}$ we have that given any $v_{\alpha}
\in {\frak L}_{\alpha}, \alpha \in \Lambda \cup \{0\}$, expressed
in the form $v_{\alpha} = v_{\alpha,{\bar 0}} + v_{{\alpha},{\bar
1}}$ with $v_{\alpha, {\bar i}} \in {\frak L}_{\bar i}, {\bar i}
\in {\hu Z}_2$, then
$$[v_{\alpha, {\bar i}},h_{\bar 0}] = \alpha(h_{\bar 0})v_{\alpha, {\bar i}}$$ for any $h_{\bar 0} \in
H_{\bar 0}$. From here, ${\frak L}_{\alpha} = ({\frak L}_{\alpha}
\cap {\frak L}_{\bar 0}) \oplus ({\frak L}_{\alpha} \cap {\frak
L}_{\bar 1})$. Hence, by denoting $${\frak L}_{\alpha, {\bar i}}
:= {\frak L}_{\alpha} \cap {\frak L}_{\bar i}$$ we can write
\begin{equation}\label{separa}
{\frak L}_{\alpha} = {\frak L}_{\alpha, {\bar 0}} \oplus {\frak
L}_{\alpha, {\bar 1}}
\end{equation}
for any $\alpha \in \Lambda \cup \{0\}$.

From the above, $$\hbox{$H_{\bar 0} = {\frak L}_{0,{\bar 0}}$ and
$H_{\bar 1} = {\frak L}_{0,{\bar 1}}$,}$$ and also
$$\hbox{${\frak L}_{\bar 0} = H_{\bar 0} \oplus (\bigoplus\limits_{\alpha \in \Lambda_{\bar 0} }{\frak L}_{\alpha,{\bar 0}}),$
\hspace{0.1cm} ${\frak L}_{\bar 1} = H_{\bar 1} \oplus
(\bigoplus\limits_{\alpha \in \Lambda_{\bar 1}} {\frak
L}_{\alpha,{\bar 1}})$,}$$

where $$\Lambda_{\bar i}=\{\alpha \in \Lambda:  {\frak L}_{\alpha,
{\bar i}} \neq 0\}$$ for any ${\bar i} \in {\mathbb Z}_2$.

Taking into account the above  expression of ${\frak L}_{\bar 0}$,
the direct character of the sum and the fact that $\alpha \neq 0$
for any $\alpha \in \Lambda$, we have that $H_{\bar 0}$ is a
maximal abelian subalgebra  of the Leibniz algebra ${\frak
L}_{\bar 0}$. Hence ${\frak L}_{\bar 0}$ is a split Leibniz
algebra respect to $H_{\bar 0}$ (see \cite{Yo3}).

\medskip

The below lemma in an immediate consequence of Super Leibniz
identity.

\begin{lemma}\label{lema1}
If $[{\frak L}_{\alpha, \bar i}, {\frak L}_{\beta, \bar j}] \neq
0$ with $\alpha, \beta \in \Lambda \cup \{0\}$ and $\bar i, \bar j
\in {\mathbb Z}_2$, then $\alpha + \beta \in \Lambda_{\bar i +
\bar j} \cup \{0\}$ and $[{\frak L}_{\alpha, \bar i}, {\frak
L}_{\beta, \bar j}] \subset {\frak L}_{\alpha+\beta, \bar i + \bar
j}$.
\end{lemma}

\smallskip

From Lemma \ref{lema1} and Equation (\ref{separa})  we also can
assert that
\begin{equation}\label{contencion}
[{\frak L}_{\alpha}, {\frak L}_{\beta}] \subset {\frak
L}_{\alpha+\beta}
\end{equation}
for any $\alpha, \beta \in \Lambda \cup \{0\}$.


\section{Connections of Roots. Decompositions}

In the following, ${\frak L}$ denotes a split Leibniz superalgebra
and ${\frak L} = H \oplus (\bigoplus\limits_{\alpha \in
\Lambda}{\frak L}_{\alpha})$ the corresponding root spaces
decomposition. We begin by developing connections of roots
techniques in this framework.

Given a linear functional $\alpha: H_{\bar 0} \to {\hu K}$, we
denote by $-\alpha: H_{\bar 0} \to {\hu K}$ the element in
$H_{\bar 0}^*$ defined by $(-\alpha)(h):=-\alpha(h)$ for all $h
\in H_{\bar 0}$. We will also denote by $-\Lambda=\{-\alpha:
\alpha \in \Lambda\}$ and by $\pm \Lambda= \Lambda \cup -
\Lambda$.

\begin{definition}\label{con}\rm
Let $\alpha \in \Lambda$ and $\beta \in \Lambda$ be two nonzero
roots. We say that $\alpha$ is {\it connected} to $\beta$ if there
exists a family $\alpha_1,\alpha_2,...,\alpha_n \in \pm \Lambda$
satisfying the following conditions:
\begin{enumerate}
\item[{\rm 1.}]$\alpha_1=\alpha.$ \item[{\rm 2.}]
$\{\alpha_1+\alpha_2,\alpha_1+\alpha_2+\alpha_3,...,\alpha_1+
\cdots + \alpha_{n-1}\} \subset \pm \Lambda.$ \item[{\rm 3.}]
$\alpha_1 +\alpha_2+ \cdots + \alpha_n \in \{\beta, -\beta\}$.
\end{enumerate}
We also say that $\{\alpha_1,...,\alpha_n\}$ is a {\it connection}
from $\alpha$ to $\beta$.
\end{definition}

Trivially, for any $\alpha \in \Lambda$ we have that the set
$\{\alpha\}$ is a connection from $\alpha$ to itself, and that in
case $-\alpha$ also belongs to $\Lambda$  the same set is a
connection from $\alpha$  to $-\alpha$.

Next, given $\alpha$ and $\beta$ two elements of $\Lambda$, we
write $\alpha \sim \beta$ when $\alpha$ is connected to $\beta$.
Note  that the observation in the above paragraph gives us $\alpha
\sim \alpha$ for any $\alpha \in \Lambda$. That is, the relation
$\sim$ is reflexive. If $\alpha \sim \beta$ there exists a
connection $\{\alpha _{1},\alpha _{2},\alpha
_{3},...,\alpha_{n-1},\alpha_{n}\}\subset \pm \Lambda$ from
$\alpha$ to $\beta$ satisfying, in particular, $\alpha_1
+\alpha_2+ \cdots + \alpha_n  = \epsilon \beta$ for some $\epsilon
\in \{1,-1\}$.
If $n=1$ then $\alpha_1= \alpha$ and $\beta \in
\{\alpha, -\alpha\}$. So $\{\beta\}$ is a connection from $\beta$
to $\alpha$. If $n > 1$, then we can  verify in a straightforward
way that
$$\{\beta ,-\epsilon\alpha_{n} ,-\epsilon\alpha_{n-1},...,
 -\epsilon\alpha_3,-\epsilon\alpha_2 \}$$ is a connection from $\beta$ to $\alpha$. That is, $\sim$ is symmetric.
 Finally, suppose $\alpha \sim \beta$ and $\beta \sim \gamma$, and
write $\{\alpha _{1},\alpha _{2},..., \alpha_{n}\}$ for a
connection from $\alpha$ to $\beta$, which satisfies $\alpha_1
+\alpha_2 \cdots + \alpha_n  = \epsilon \beta$ for some $\epsilon
\in \{1,-1\}$; and $\{\beta_1,\beta_2,..., \beta_{m}\}$ for a
connection from $\beta$ to $\gamma$. If $m=1$, then $\gamma \in
\pm \{\beta,-\beta\}$ and so $\{\alpha _{1},\alpha _{2},...,
\alpha_{n}\}$ is a connection from $\alpha$ to $\gamma$. If $m>
1$, then it is easy to show that $\{\alpha _{1},\alpha _{2},...,
\alpha_{n},\epsilon \beta _{2},..., \epsilon \beta_{m}\}$ is a
connection from $\alpha$ to $\gamma$. That is, $\sim$ is
transitive. Summarizing, (see also \cite[Proposition 2.1]{Yo3}),
we can assert:

\begin{proposition}\label{pro1}
The relation $\sim$ in $\Lambda$, defined by $\alpha \sim \beta$
if and only if $\alpha$ is connected to $\beta$ is an equivalence
relation.
\end{proposition}

By Proposition \ref{pro1} the connection relation is an
equivalence relation in $\Lambda$ and so we can consider the
quotient set
$$\Lambda / \sim := \{[\alpha]: \alpha \in \Lambda\},$$
becoming  $[\alpha]$  the set of nonzero roots of ${\frak L}$
which are connected to $\alpha$. By the definition of $\sim$, it
is clear that if $\beta \in [\alpha]$ and $-\beta \in \Lambda$
then $-\beta \in [\alpha]$.

\medskip

Our next goal  is to associate an ideal $I_{[\alpha]}$ of  ${\frak
L}$ to any $[\alpha]$. Fix $\alpha \in \Lambda$, we start by
defining the set $H_{[\alpha]} \subset H$ as follows
$$H_{[\alpha]} := span_{\hu K}\{[{\frak L}_{\beta},{\frak L}_{-\beta}]: \beta \in [\alpha]\}=$$
$$= (\sum\limits_{\beta \in [\alpha]}([{\frak L}_{\beta,{\bar
0}},{\frak L}_{-\beta,{\bar 0}}]+[{\frak L}_{\beta,{\bar
1}},{\frak L}_{-\beta,{\bar 1}}])) \oplus (\sum\limits_{\beta \in
[\alpha]}([{\frak L}_{\beta,{\bar 0}},{\frak L}_{-\beta,{\bar
1}}]+[{\frak L}_{\beta,{\bar 1}},{\frak L}_{-\beta,{\bar 0}}]))
\subset$$
$$\subset H_{\bar 0} \oplus H_{\bar 1},$$ last equality being
consequence of Equation (\ref{contencion}). Next, we define
$$V_{[\alpha]} := \bigoplus\limits_{\beta \in
[\alpha]}{\frak L}_{\beta} = \bigoplus\limits_{\beta \in
[\alpha]}({\frak L}_{\beta,{\bar 0}} \oplus {\frak L}_{\beta,
{\bar 1}}).$$ Finally, we denote by ${\frak L}_{[\alpha]}$ the
(graded) subspace given by the direct sum of the two subspaces
above, that is,
 $${\frak L}_{[\alpha]} := H_{[\alpha]} \oplus V_{[\alpha]}.$$

\begin{proposition}\label{pro2}
For any $\alpha \in \Lambda$, the linear subspace ${\frak
L}_{[\alpha]}$ is a Leibniz supersubalgebra of ${\frak L}$.


\end{proposition}

\begin{proof}
We have to check that ${\frak L}_{[\alpha]}$ satisfies
$$[{\frak L}_{[\alpha]},{\frak
L}_{[\alpha]}] \subset {\frak L}_{[\alpha]}.$$
 Taking into account $H = {\frak L}_0$ and Equation (\ref{contencion}), we
have
\begin{equation}\label{cero}
[{\frak L}_{[\alpha]}, {\frak L}_{[\alpha]}] = [H_{[\alpha]}
\oplus V_{[\alpha]}, H_{[\alpha]} \oplus V_{[\alpha]}] \subset
V_{[\alpha]} + \sum\limits_{\beta, \delta \in [\alpha]}[{\frak
L}_{\beta},{\frak L}_{\delta}].
\end{equation}
If $\delta = -\beta$ then
\begin{equation}\label{uno}
[{\frak L}_{\beta}, {\frak L}_{\delta}] \subset H_{[\alpha]}.
\end{equation}
If $\delta \neq -\beta$, by Equation (\ref{contencion}) we have
that in case $[{\frak L}_{\beta}, {\frak L}_{\delta}] \neq 0$ then
$\beta + \delta \in \Lambda$. From here, $\{\beta,\delta\}$ is a
connection from $\beta$ to $\beta + \delta$. Hence $\beta + \delta
\in [\alpha]$ and so
\begin{equation}\label{dos}
[{\frak L}_{\beta},{\frak L}_{\delta}] \subset V_{[\alpha]}.
\end{equation}
From Equations (\ref{cero}), (\ref{uno}) and (\ref{dos}) we
conclude $[{\frak L}_{[\alpha]}, {\frak L}_{[\alpha]}] \subset
{\frak L}_{[\alpha]}$.
\end{proof}

\begin{proposition}\label{nuevapro}
For any $[\alpha ]\neq [\beta ]$ we have $[{\frak
L}_{[\alpha]},{\frak L}_{[\beta]}] = 0.$
\end{proposition}

\begin{proof}
 We have $$[{\frak L}_{[\alpha]}, {\frak
L}_{[\beta]}] = [H_{[\alpha]} \oplus V_{[\alpha]}, H_{[\beta]}
\oplus V_{[\beta]}] \subset $$
\begin{equation}\label{cuatro}
\subset [H_{[\alpha]}, V_{[\beta]}] + [V_{[\alpha]}, H_{[\beta]}]
+ [V_{[\alpha]}, V_{[\beta]}].
\end{equation}
Consider the above third summand $[V_{[\alpha]},V_{[\beta]}]$ and
suppose there exist $\eta \in [\alpha]$ and $\gamma \in [\beta]$
such that $[{\frak L}_{\eta}, {\frak L}_{\gamma}] \neq 0$. As
necessarily $\eta \neq -\gamma$, then $\eta + \gamma \in \Lambda$.
So $\{\eta,  \gamma, -\eta\}$ is a connection from $\eta$ to
$\gamma$. By the transitivity of the connection relation we have
$\alpha \sim \beta$, a contradiction. Hence $[{\frak L}_{\eta},
{\frak L}_{\gamma}] = 0$ and so
\begin{equation}\label{nueve}
[V_{[\alpha]}, V_{[\beta]}] = 0.
\end{equation}
Consider now the second summand $[V_{[\alpha]}, H_{[\beta]}]$ from
(\ref{cuatro}) and suppose exist $\eta \in [\alpha], \gamma \in
[\beta]$ such that $[{\frak L}_{\eta}, [{\frak L}_{\gamma}, {\frak
L}_{-\gamma}]] \neq 0$. Hence, there exist $\bar{i}, \bar{j},
\bar{k} \in {\mathbb Z}_2$ such that
$$[{\frak L}_{\eta, \bar{i}},[{\frak L}_{\gamma,\bar{j}}, {\frak
L}_{-\gamma, \bar{k}}]]  \neq 0.$$ By Super Leibniz identity we
get either $[{\frak L}_{\eta, \bar{i}},{\frak L}_{\gamma,\bar{j}}]
\neq 0$ or $[{\frak L}_{\eta, \bar{i}},{\frak L}_{-\gamma,
\bar{k}}] \neq 0$. From here $[V_{[\alpha]},V_{[\beta]}] \neq 0$
in any case, what contradicts Equation (\ref{nueve}).

Finally, consider the first summand $[H_{[\alpha]}, V_{[\beta]}]$
from (\ref{cuatro}) and suppose there exist $\eta \in [\alpha],
\gamma \in [\beta]$ such that $[[{\frak L}_{\eta}, {\frak
L}_{-\eta}], {\frak L}_{\gamma}] \neq 0$. We have as above that we
can take $\bar{i}, \bar{j}, \bar{k} \in {\mathbb Z}_2$ satisfying
$$[[{\frak L}_{\eta, {\bar i}},{\frak L}_{-\eta,{\bar j}}], {\frak
L}_{\gamma,{\bar k}}] \neq 0,$$ and that by Super Leibniz identity
 either $[{\frak L}_{\eta,\bar{i}},{\frak L}_{\gamma,\bar{k}
}] \neq 0$ or $[{\frak L}_{-\eta,\bar{j}},{\frak
L}_{\gamma,\bar{k}}] \neq 0$. From here
$[V_{[\alpha]},V_{[\beta]}] \neq 0$ in any case, what contradicts
Equation (\ref{nueve}). Hence
$$[V_{[\alpha]}, V_{[\beta]}] = 0.$$ By Equation
(\ref{cuatro}) we conclude $[{\frak L}_{[\alpha]},{\frak
L}_{[\beta]}] = 0$.
\end{proof}

Proposition \ref{pro2} allows us to assert that for any $\alpha
\in \Lambda, {\frak L}_{[\alpha]}$ is a  Leibniz supersubalgebra
of ${\frak L}$ that we call the Leibniz supersubalgebra of ${\frak
L}$ {\it associated} to $[\alpha]$.

\begin{theorem}\label{teo1}
The following assertions hold.
\begin{enumerate}
\item[{\rm 1.}]  For any $\alpha \in \Lambda$, the Leibniz
supersubalgebra
\[
{\frak L}_{[\alpha]} = H_{[\alpha]} \oplus V_{[\alpha]}
\]
of ${\frak L}$ associated to $[\alpha]$ is an ideal of ${\frak
L}$. \item[{\rm 2.}] If ${\frak L}$ is simple, then there exists a
connection between any two nonzero roots of ${\frak L}$ and $H =
\sum\limits_{\beta \in \Lambda} [{\frak L}_{\beta}, {\frak
L}_{-\beta}]$.
\end{enumerate}
\end{theorem}

\begin{proof}
1. We have to check that
\[
[{\frak L}_{[\alpha]}, {\frak L}] + [{\frak L}, {\frak
L}_{[\alpha]}] \subset {\frak L}_{[\alpha]}.
\]
Since ${\frak L}_{[\alpha]} = H_{[\alpha]} \oplus V_{[\alpha]}$
and ${\frak L} = H\oplus \big(\bigoplus\limits_{\beta \in
\Lambda}{\frak L}_{\beta }\big)$, we are going to consider several
steps. First note that
\[
[H_{[{\alpha}]}, H] + [H, {H}_{[{\alpha}]}]  \subset [H,H] = 0.
\]
Next, given $\beta \in \Lambda $ an application of Propositions
\ref{pro2} and \ref{nuevapro} gives
\[
[{\frak L}_{[{\alpha}]}, {\frak L}_\beta] + [{\frak L}_\beta,
{\frak L}_{[{\alpha}]}] \subset {\frak L}_{[{\alpha}]}.
\]
Therefore,
$$[{\frak L}_{[{\alpha}]}, {\frak L}] + [{\frak L}, {\frak
L}_{[{\alpha}]}]  \subset {\frak L}_{[{\alpha}]}.$$

\smallskip

2. The simplicity of ${\frak L}$ applies to get that ${\frak
L}_{[\alpha]} \in \{ {\frak I}, {\frak L} \}$ for any $\alpha \in
\Lambda$. If ${\frak L}_{[{\alpha_0}]}={\frak L}$ for some
$\alpha_0 \in \Lambda$ then ${[{\alpha_0}]} = \Lambda$; which
implies that any nonzero root of ${\frak L}$ is connected to
$\alpha_0$, and therefore, any nonzero roots of ${\frak L}$ are
connected. Assume now that ${\frak L}_{[\alpha]} = {\frak I}$ for
every $\alpha \in \Lambda$. Then $[\alpha] =[\beta]$ for any
$\alpha, \beta \in {\frak L}$. Since
$\Lambda=\bigcup\limits_{[\gamma]\in \Lambda \diagup \sim}
[\gamma]$ we get $[\alpha]=\Lambda$. That is,  all nonzero roots
of ${\frak L}$ are connected. To finish, observe that the fact $H
= \sum\limits_{\beta \in \Lambda} [{\frak L}_{\beta}, {\frak
L}_{-\beta}]$  follows in any case.
\end{proof}

\smallskip

\begin{notation}\rm
Let us denote $$H_{\Lambda}:=\sum\limits_{\beta \in \Lambda}
[{\frak L}_{\beta}, {\frak L}_{-\beta}]
$$
In what follows, we will use the  terminology $I_{[\alpha]}:=
{\frak L}_{[\alpha]}$ where ${\frak L}_{[\alpha]}$ is one of the
ideals of ${\frak L}$ described in Theorem \ref{teo1}-1.
\end{notation}

\begin{theorem} \label{teo2}
For a vector space complement $\mathcal{U}$ of $H_{\Lambda}$ in
$H$, it follows
\[
{\frak L} = \mathcal{U} + \sum_{[\alpha] \in \Lambda/\sim}
I_{[\alpha]}.
\]
Moreover,
\[
[I_{[\alpha]},I_{[\beta]}] = 0,
\]
whenever $[\alpha] \neq [\beta]$.
\end{theorem}
\begin{proof}
From
\[
{\frak L} = H \oplus \big( \bigoplus\limits_{\alpha \in
\Lambda}{\frak L}_\alpha \big )= (\mathcal{U}\oplus
{H_{\Lambda}})\oplus (\bigoplus_{\alpha \in \Lambda} {\frak
L}_{\alpha}),
\]
it follows
\[
\bigoplus\limits_{\alpha \in \Lambda}{\frak L}_\alpha =
\bigoplus\limits_{[\alpha] \in \Lambda^{/\sim}}V_{[\alpha]}, \quad
\quad H_{\Lambda} = \sum_{[\alpha] \in
\Lambda^{/\sim}}H_{[\alpha]},
\]
which implies
\[
{\frak L} = (\mathcal{U}\oplus H_{\Lambda})\oplus
(\bigoplus_{\alpha \in \Lambda}{{\frak L}_\alpha})= \mathcal{U} +
\sum\limits_{[\alpha] \in \Lambda^{/\sim}} I_{[\alpha]},
\]
where each $I_{[\alpha]}$ is an ideal of ${\frak L}$ by Theorem
\ref{teo1}. Now, given $[\alpha] \neq [\beta]$, the assertion
\[
[I_{[\alpha]},I_{[\beta]}] = 0,
\]
follows of Proposition \ref{nuevapro}.
\end{proof}

\begin{corollary}\label{coro1}
If ${\mathcal Z}({\frak L}) = 0$ and ${\frak L} = [{\frak
L},{\frak L}]$, then
\[
{\frak L} =\bigoplus\limits_{[\alpha] \in \Lambda/\sim}
I_{[\alpha]}.
\]
\end{corollary}

\begin{proof}
Since $[{\frak L},{\frak L}] = {\frak L}$, Theorem \ref{teo2}
applies to get
\[
\left[\mathcal{U} + \sum_{[\alpha] \in \Lambda/\sim} I_{[\alpha]},
\, \, \, \mathcal{U} + \sum_{[\alpha] \in \Lambda/\sim}
I_{[\alpha]}\right] = \mathcal{U} + \sum_{[\alpha] \in
\Lambda/\sim} I_{[\alpha]}.
\]
Keeping in mind that
\[
\mathcal{U} \subset H \quad \mbox{ and } \quad [
I_{[\alpha]},I_{[\beta]}]  = 0, \, \, \mbox{ if } \, \, [\alpha]
\neq [\beta],
\]
we get $\mathcal{U} = 0$, that is,
\[
{\frak L} = \sum_{[\alpha] \in \Lambda/\sim} I_{[\alpha]}.
\]
To finish, we show the direct character of the sum. Given $x\in
I_{[\alpha]} \cap \sum\limits_{\tiny{\begin{array}{c}
  [\beta] \in \Lambda / \sim  \\
\beta \nsim \alpha\\
\end{array}}} I_{[\beta]}$
then using again the equation $\, \, [I_{[\alpha]},I_{[\beta]}]  =
0$, for $[\alpha] \neq [\beta]$, we obtain 
\begin{align*}
\big[x, I_{[\alpha]}\big] + \left[x, \sum_{\tiny{\begin{array}{c}
[\beta] \in \Lambda / \sim  \\
\beta \nsim \alpha\\
\end{array}}} I_{[\beta]}
\right] & = 0,
\\
&
\\
\big[I_{[\alpha]},x\big] + \left[ \sum_{\tiny{\begin{array}{c}
[\beta] \in \Lambda / \sim  \\
\beta \nsim \alpha\\
\end{array}}} I_{[\beta]}
,x\right] & = 0.
\end{align*}
It implies $[x, {\frak L}]+ [{\frak L}, x]  = 0$, that is, $x \in
{\mathcal Z}({\frak L}) = 0$. Thus $x = 0$, as desired.
\end{proof}

\section{Split Leibniz superalgebras of maximal length. The simple components}

In this section we focus on the simplicity of split Leibniz
superalgebras by centering our attention in those of  maximal
length. This terminology is taking borrowed from the theory of
gradations of Lie and Leibniz algebras, (see \cite{Alberiogradu,
Diamante, Gomez1, Gomez2}). See also \cite{Diamante, Yo1,
Yotriple1, Yo2, Yo3, Camacho, Neeb} for examples.

\begin{definition}\rm
We say that a split Leibniz superalgebra ${\frak L}$ is of {\it
maximal length} if $\dim ({{\frak L}}_{\alpha, \bar{i}}) \in \{0,
1\}$ for any $\alpha \in \Lambda,$ $ \bar{i} \in {\hu Z}_2.$
\end{definition}

Our target is to characterize the simplicity of ${\frak L}$ in
terms of connectivity properties in $\Lambda$. Therefore  we would
like to attract attention to  the definition of simple Leibniz
superalgebra given in Definition \ref{Defsimple}, and the previous
discussion.

\medskip

The following lemma is consequence of the fact that the set of
multiplications by elements in $H_{\bar 0}$ is a commuting set of
diagonalizable endomorphisms and $I$ is invariant under this set.

\begin{lemma}\label{lema5}
Let ${\frak L}=H \oplus (\bigoplus\limits_{\alpha \in \Lambda}
{\frak L}_{\alpha})$ be a split  Leibniz superalgebra. If $I$ is
an ideal of ${\frak L}$ then $I= (I\cap H )\oplus
(\bigoplus\limits_{\alpha \in \Lambda} (I \cap {\frak
L}_{\alpha})).$
\end{lemma}

\medskip

From now on ${\frak L} = H \oplus (\bigoplus\limits_{\alpha \in
\Lambda}{\frak L}_{\alpha})$ denotes a split Leibniz superalgebra
of maximal length without further mention. We begin by observing
that in this case Lemma \ref{lema5} allows us to assert that given
any nonzero (graded) ideal $I=I_{\bar 0}\oplus I_{{\bar 1}}$ of
${\frak L}$ then

$$I = (I \cap H) \oplus (\bigoplus\limits_{\alpha \in \Lambda} I
\cap {\frak L}_{\alpha})=$$ $$(I \cap H) \oplus
(\bigoplus\limits_{\alpha \in \Lambda} ((I_{\bar{0}} \cap {\frak
L}_{\alpha, {\bar 0}}) \oplus (I_{{\bar 1}} \cap {\frak
L}_{\alpha, {\bar 1}})))=$$

\begin{equation}\label{max}
(I \cap H) \oplus (\bigoplus\limits_{\alpha \in
\Lambda_{\bar{0}}^{I}} {\frak L}_{\alpha, \bar{0}}) \oplus
(\bigoplus\limits_{\beta \in \Lambda_{\bar{1}}^{I}} {\frak
L}_{\beta, \bar{1}})
\end{equation}
where $\Lambda_{\bar{i}}^{I} := \{\gamma \in \Lambda : I_{\bar i}
\cap {\frak L}_{\gamma, \bar{i}} \neq 0\}$ for $\bar{i} \in {\hu
Z}_2$. In the particular, (an important) case $I = {\frak I}$, we
get

\begin{equation}\label{I}
{\frak I} = ({\frak I} \cap H) \oplus (\bigoplus\limits_{\alpha
\in \Lambda_{\bar{0}}^{{\frak I}}} {\frak L}_{\alpha, \bar{0}})
\oplus (\bigoplus\limits_{\beta \in \Lambda_{\bar{1}}^{{\frak I}}}
{\frak L}_{\beta, \bar{1}})
\end{equation}
with $$\Lambda_{\bar{i}}^{\frak I} := \{\gamma \in \Lambda :
{\frak I}_{\bar{i}} \cap {\frak L}_{\gamma, \bar{i}} \neq
0\}=\{\gamma \in \Lambda : 0\neq {\frak L}_{\gamma, \bar{i}}
\subset {\frak I}_{\bar{i}} \},$$ $\bar{i} \in {\hu Z}_2$.

\smallskip

From here, we can write
\begin{equation}\label{separa2}
\Lambda = \underbrace{(\Lambda_{\bar{0}}^{{\frak I}}\dot{\cup}
\Lambda_{\bar{0}}^{\neg{\frak I}})}_{\Lambda_{\bar 0}} \cup
\underbrace{(\Lambda_{\bar{1}}^{{\frak I}}\dot{\cup}
\Lambda_{\bar{1}}^{\neg{\frak I}})}_{\Lambda_{\bar 1}}
\end{equation}
where $$\hbox{$\Lambda_{\bar{i}}^{\neg{\frak I}} := \{\gamma \in
\Lambda: {\frak L}_{\gamma,\bar{i}} \neq 0$ and ${\frak I}_ {\bar
i} \cap {\frak L}_{\gamma,\bar{i}}  = 0\}$}$$ for $\bar{i} \in
{\hu Z}_2$. We will also denote
$$\hbox{$\Lambda^{\Upsilon}:=\Lambda_{\bar{0}}^{{\Upsilon}} \cup
\Lambda_{\bar{1}}^{{\Upsilon}}$, for $\Upsilon \in \{{\frak I},
\neg {\frak I}\}$.}$$

Hence, we can write
\begin{equation}\label{Llargo}
{\frak L} = (H_{\bar{0}} \oplus H_{\bar{1}}) \oplus
(\bigoplus\limits_{\alpha \in \Lambda_{\bar{0}}^{ {\frak I}}}
{\frak L}_{\alpha, \bar{0}}) \oplus
 (\bigoplus\limits_{\beta \in
\Lambda_{\bar{0}}^{\neg{\frak I}}}
 {\frak L}_{\beta, \bar{0}})
\oplus (\bigoplus\limits_{\gamma \in \Lambda_{\bar{1}}^{ {\frak
I}}} {\frak L}_{\gamma, \bar{1}}) \oplus
 (\bigoplus\limits_{\delta \in
\Lambda_{\bar{1}}^{\neg{\frak I}}}
 {\frak L}_{\delta, \bar{1}}).
\end{equation}

\smallskip

\begin{example}\label{example3}\rm
Consider the 5-dimensional split Leibniz superalgebra  $${\frak
L}=H \oplus {\frak L}_{\alpha} \oplus {\frak L}_{-\alpha}\oplus
{\frak L}_{\beta}\oplus {\frak L}_{-\beta}$$ given in Example
\ref{example1}, where ${\frak L}_{\bar 0}=H \oplus {\frak
L}_{\alpha} \oplus {\frak L}_{-\alpha}$ and ${\frak L}_{\bar
1}={\frak L}_{\beta}\oplus {\frak L}_{-\beta}$. This is a split
Leibniz superalgebra of maximal length such  that ${\frak
I}=\langle e_1, e_2 \rangle$. Hence $\Lambda^{\neg \frak I}_{ \bar
0} = \{\alpha,-\alpha\}, \Lambda^{\frak I}_{\bar 1} =
\{\beta,-\beta\}$ and $\Lambda^{\frak I}_{ \bar 0}=\Lambda^{\neg
\frak I}_{ \bar 1}= \emptyset$.
\end{example}

\smallskip

\begin{example}\label{example4}\rm
Consider the $(n+4)$-dimensional split Leibniz superalgebra
$${\frak L}=H \oplus {\frak L}_{\alpha} \oplus {\frak
L}_{-\alpha}\oplus (\bigoplus\limits_{k=0}^{n}{\frak
L}_{\beta_k})$$ given in Example \ref{example2}, where ${\frak
L}_{\bar 0}=H \oplus {\frak L}_{\alpha} \oplus {\frak
L}_{-\alpha}$ and ${\frak L}_{\bar
1}=\bigoplus\limits_{k=0}^{n}{\frak L}_{\beta_k}$. This is a split
Leibniz superalgebra of maximal length satisfying  ${\frak
I}=\langle e_0,..., e_n \rangle$. From here, $\Lambda^{\neg \frak
I}_{\bar 0}=\{\alpha,-\alpha\}$ and $\Lambda^{\frak I}_{\bar
1}=\{e_0,...,e_n\}$.
\end{example}

\smallskip

\begin{remark}\label{re3.1}\rm
Since our aim is this section to characterize the simplicity of
${\frak L}$, in terms of connectivity of roots, Theorem
\ref{teo1}-2
gives us that we have to center our attention in those split
Leibniz superalgebras satisfying  $H = \sum\limits_{\beta \in
\Lambda} [{\frak L}_{\beta}, {\frak L}_{-\beta}]$. This is for
instance the case whence ${\frak L}=[{\frak L},{\frak L}]$.  We
would like to note that if $H = \sum\limits_{\beta \in \Lambda}
[{\frak L}_{\beta}, {\frak L}_{-\beta}]$, then the decomposition
given by Equation (\ref{Llargo}) and Equation (\ref{equi}) show
$$\small{H} =$$
\begin{equation}\label{equsimple}
\small{(\underbrace{\sum\limits_{\alpha \in \Lambda_{\bar 0}^{\neg
{\frak I}}} [{\frak L}_{-\alpha,\bar{0}},{\frak
L}_{\alpha,\bar{0}}]+ \sum\limits_{\alpha \in \Lambda_{\bar
1}^{\neg {\frak I}}}[{\frak L}_{-\alpha,\bar{1}},{\frak
L}_{\alpha,\bar{1}}]}_{H_{\bar{0}}}) \oplus
(\underbrace{\sum\limits_{\alpha \in \Lambda_{\bar 0}^{\neg {\frak
I}}} [{\frak L}_{-\alpha,\bar{1}},{\frak L}_{\alpha,\bar{0}}] +
\sum\limits_{\alpha \in \Lambda_{\bar 1}^{\neg {\frak I}}} [{\frak
L}_{-\alpha,\bar{0}},{\frak L}_{\alpha,\bar{1}}]}_{H_{\bar{1}}})}.
\end{equation}
\end{remark}

\medskip

Now, observe that the concept of connectivity of nonzero roots
given in Definition \ref{con} is not strong enough to detect if a
given $\alpha \in \Lambda$ belongs to $\Lambda_{\bar i}^{{{\frak
I}}}$ or to $\Lambda_{\bar i}^{\neg{{\frak I}}}$ for some ${\bar
i} \in {\mathbb Z}_2$. Consequently we lose the information
respect to whether a given root space ${{\frak L}}_{\alpha}$
intersects to ${{\frak I}}$ in a non-trivial way or not, which is
fundamental to study the simplicity of ${{\frak L}}$. So, we are
going to refine the previous  concept of connections of non-zero
roots.

\begin{definition}\label{def_cone_S}\rm
Let $\alpha \in \Lambda_{\bar i}^{\Upsilon}$ and $ \beta \in
\Lambda_{\bar j}^{\Upsilon}$ with $\Upsilon \in \{ {\frak I},
\neg{\frak I} \}$ and ${\bar i}, {\bar j} \in {\mathbb Z}_2$. We
say that $\alpha$ is {\it ${\neg{\frak I}}$-connected} to $\beta,$
denoted by $\alpha \sim_{\neg{\frak I}} \beta,$ if either $\beta
\in \{\alpha,-\alpha\}$ or there exists a family of nonzero roots
$\alpha_1,\alpha _2,...,\alpha_n $ such that $$\alpha_k \in
\Lambda_{{\bar i}_k}^{\neg{\frak I}}$$ for some ${\bar i}_k \in
{\mathbb Z }_2$ and for any $k=2,...,n$; and such that
\begin{itemize}
\item[1.] $\alpha_1 = \alpha$.
\item[2.] $ \alpha_1+\alpha_2 \in  \Lambda_{{\bar i}+ {\bar
i_2}}^{\Upsilon} ,$\\
$ \alpha_1+\alpha_2 +\alpha_3\in  \Lambda_{{\bar i}+ {\bar
i_2}+{\bar i_3}}^{\Upsilon} ,$\\
$\vdots$\\
 $\alpha_1+\alpha_2+\cdots+\alpha_{n-1} \in
\Lambda_{{\bar i}+ {\bar i_2}+\cdots+{\bar i_{n-1}}}^{\Upsilon},
$\\
$\alpha_1+\alpha_2+\cdots+\alpha_{n-1}+\alpha_n \in \Lambda_{{\bar
i}+ {\bar i_2}+\cdots+{\bar i_{n-1}}+ {\bar i_n}}^{\Upsilon}$

and \item[3.]$\alpha_1 + \alpha_2+\cdots + \alpha_{n-1}+\alpha_n
=\beta$ and ${\bar i}+ {\bar i_2}+\cdots+{\bar i_{n-1}}+ {\bar
i_n}={\bar j}$.
\end{itemize}
The set $\{\alpha_1,\alpha_2,..., \alpha_n\}$ is called a {\it
${\neg{\frak I}}$-connection} from $\alpha$ to $\beta$.
\end{definition}

\smallskip


Let us introduce the notion of root-multiplicativity in the
framework of split Leibniz superalgebras of maximal length, in a
similar way to the ones for split Lie algebras, split Lie
superalgebras and  split Leibniz algebras among other split
algebraic structures, (see \cite{Yo1, Yo5, Yo2, Yo4, Yo3}  for
these notions and examples).

\begin{definition}\label{defmulti}
We say that a split Leibniz superalgebra of maximal length ${\frak
L}$ is {\it root-multiplicative} if the below conditions hold.
\begin{enumerate}
\item Given $\alpha \in {\Lambda}_{\bar{i}}^{{\neg {\frak I}}},
\beta \in {\Lambda}_{\bar{j}}^{{\neg {\frak I}}}$ such that
$\alpha+\beta \in \Lambda$ then $[{\frak L}_{\alpha,\bar{i}},
{\frak L}_{\beta, \bar{j}}]\neq 0.$ \item Given $\alpha \in
{\Lambda}_{\bar{i}}^{\neg {\frak I}}$ and $\gamma \in
{\Lambda}_{\bar{j}}^{{\frak I}}$ such that $\alpha+\gamma \in
{\Lambda}^{{\frak I}}$ then $[{\frak L}_{\gamma, \bar{j}},{\frak
L}_{\alpha, \bar{i}}] \neq 0.$
\end{enumerate}
\end{definition}

\begin{notation}\rm
We will say that $\Lambda^{\Upsilon}$, $\Upsilon \in \{{\frak I},
\neg {\frak I}\}$, has all of its roots {\it ${\neg {\frak
I}}$-connected} if for any ${\bar i}, {\bar j} \in {\mathbb Z}_2$
we have that $\Lambda^{\Upsilon}_{\bar i}$ has all of its roots
connected to any root in $\Lambda^{\Upsilon}_{\bar j}$.
\end{notation}

\begin{proposition}\label{nueva12}
Suppose $H = \sum\limits_{\beta \in \Lambda} [{\frak L}_{\beta},
{\frak L}_{-\beta}]$ and ${\frak L}$ is root-multiplicative.
 If  $\Lambda^{\neg {\frak
I}}$  has all of its roots ${\neg {\frak I}}$-connected and
$|\Lambda^{\neg {\frak I}}|>2$,  then any nonzero ideal $I$ of
${\frak L}$ such that $I \nsubseteq H + {\frak I}$
satisfies that  $I={\frak L}$
\end{proposition}

\begin{proof}
Since $I \nsubseteq H + {\frak I}$, there exists $\alpha_0 \in
\Lambda^{\neg J}_{{\bar i}_0}$   such that
\begin{equation}\label{I2}
0\neq {\frak L}_{\alpha_0,{\bar i}_0} \subset I.
\end{equation}
for some ${\bar i}_0 \in {\hu Z}_2$, (see Equation
(\ref{Llargo})). Given now any  $\alpha \in \Lambda^{\neg {\frak
I}}_{\bar j} \setminus \{\alpha_0,-\alpha_0\}$, being then  $0
\neq {\frak L}_{\alpha, {\bar j}}$, the fact that $\alpha_0$ and
$\alpha$ are ${\neg {\frak I}}$-connected
gives us a ${\neg {\frak I}}$-connection
$\{\delta_1,\delta_2,...,\delta_n\} $ from $\alpha_0$ to $\alpha$
such that $$\delta_1 = \alpha_0 \in \Lambda_{{\bar
i}_0}^{\neg{\frak I}}, \hspace{0.1cm} \delta_k \in \Lambda_{{\bar
i}_k}^{\neg{\frak I}}\hspace{0.05cm} {\rm for}\hspace{0.05cm}
k=2,...,n,$$
$$\delta_1+\delta_2 \in \Lambda_{{\bar i_0}+ {\bar i_2}}^{\neg
{\frak I}} ,..., \delta_1+\delta_2+\cdots+\delta_{n-1} \in
\Lambda_{{\bar i_0}+ {\bar i_2}+\cdots+{\bar i_{n-1}}}^{\neg{\frak
I}},\hspace{0.1cm} {\rm and}$$
$$\delta_1+\delta_2+\cdots+\delta_{n-1}+\delta_n \in
\Lambda_{{\bar i_0}+ {\bar i_2}+\cdots+{\bar i_{n-1}}+ {\bar
i_n}}^{\neg{\frak I}}$$ with $\delta_1 + \delta_2+\cdots +
\delta_{n-1}+\delta_n =\alpha$ and ${\bar i_0}+ {\bar
i_2}+\cdots+{\bar i_{n-1}}+ {\bar i_n}={\bar j}$.

\smallskip

Consider $\delta_1, \delta_2$ and $\delta_1 + \delta_2$. Since
$\delta_1 = \alpha_0 \in \Lambda_{{\bar i}_0}^{\neg {\frak I}}$,
$\delta_2 \in \Lambda_{{\bar i}_2}^{\neg {\frak I}}$ and $\delta_1
+ \delta_2\in \Lambda$, the root-multiplicativity and maximal
length of ${\frak L}$ show $$0 \neq [{\frak L}_{\delta_1,{{\bar
i}_0}}, {\frak L}_{\delta_2,{{\bar i}_2}}] = {\frak
L}_{\delta_1+\delta_2,{\bar i}_0}+{{\bar i}_2},$$ and by Equation
(\ref{I2}) $$0 \neq {\frak L}_{\delta_1+\delta_2,{\bar i}_0+{\bar
i}_2} \subset I.$$

We can argue in a similar way from $\delta_1+\delta_2,\delta_3$
and $\delta_1+\delta_2+\delta_3$. That is, $\delta_1+\delta_2 \in
\Lambda_{{\bar i}_0+{\bar i}_2}^{\neg {\frak I}}, \delta_3 \in
\Lambda_{{\bar i}_3}^{\neg {\frak I}}$ and
$\delta_1+\delta_2+\delta_3 \in \Lambda.$ Hence $$0 \neq [{\frak
L}_{\delta_1+\delta_2,{\bar i}_0+{\bar i}_2}, {\frak
L}_{\delta_3,{\bar i}_3}] = {\frak
L}_{\delta_1+\delta_2+\delta_3,{\bar i}_0+{\bar i}_2+{\bar
i}_3},$$ and by the above $$0 \neq {\frak L}_{\delta_1 + \delta_2
+ \delta_3, {\bar i}_0+{\bar i}_2+{\bar i}_3} \subset I.$$

Following this process with the ${\neg {\frak I}}$-connection
$\{\delta_1,...,\delta_n\}$ we obtain that
$$0 \neq{\frak L}_{\delta_1+\delta_2+\delta_3+\cdots+\delta_n,{\bar i_0}+ {\bar
i_2}+\cdots+{\bar i_{n-1}}+ {\bar i_n}} \subset I$$ and so ${\frak
L}_{\alpha,{\bar j}} \subset I$. That is, for any $\alpha \in
\Lambda^{\neg {\frak I}}_{\bar j}\setminus \{\alpha_0,-\alpha_0\}$
and for each ${\bar j} \in {\mathbb Z}_2$, we have
\begin{equation}\label{ootraa}
{\frak L}_{\alpha,{\bar j}} \subset I.
\end{equation}

Let us now verify that in case $0\neq {\frak L}_{\epsilon
\alpha_0, \bar k}$ for some $\epsilon \in \{1,-1\}$ and ${\bar k}
\in {\mathbb Z}_2$, we have $0\neq {\frak L}_{\epsilon \alpha_0,
\bar k} \subset I.$ Indeed, since $|\Lambda^{\neg {\frak I}}|>2$,
we can take $\beta \in \Lambda^{\neg \frak I}_{\bar i}$, for some
$\bar i \in {\mathbb Z}_2$, such that $\beta \notin
\{\alpha_0,-\alpha_0\}$ and, by Equation (\ref{ootraa}),
satisfying $0 \neq {\frak L}_{\alpha,{\bar j}} \subset I$. Hence
we can argue as above  with the root-multiplicativity and maximal
length of ${\frak L}$ from $\beta$ instead of $\alpha_0$, to get
that in case some $\epsilon \alpha_0 \in \Lambda^{\neg \frak
I}_{\bar k}$ for some $\epsilon \in \{1,-1\}$ and ${\bar k} \in
{\mathbb Z}_2$, then
\begin{equation}\label{lara}
0 \neq {\frak L}_{\epsilon \alpha_0, \bar k}\subset I.
\end{equation}
Since $H = \sum\limits_{\beta \in \Lambda} [{\frak L}_{\beta},
{\frak L}_{-\beta}]$, Remark \ref{re3.1} and equations
(\ref{ootraa}), (\ref{lara})  give us
\begin{equation}\label{lah}
H \subset I.
\end{equation}
Now, for any $\Upsilon \in \{{\frak I}, \neg{\frak I}\}$ and
${\bar k} \in {\mathbb Z}_2$, given  any $\gamma \in
\Lambda^{\Upsilon}_{\bar k}$ since $\gamma \neq 0$ we have ${\frak
L}_{\gamma,{\bar k}}=[ {\frak L}_{\gamma, {\bar k}},H_{\bar 0}]$
and then Equation (\ref{lah}) shows ${\frak L}_{\gamma, {\bar k}}
\subset I$. The decomposition of ${\frak L}$ in Equation
(\ref{Llargo}) finally gives us $I={\frak L}$.
\end{proof}

Let us introduce an interesting notion related to a split Leibniz
superalgebra of maximal length ${\frak L}$.
We wish to distinguish  those  elements of ${\frak L}$ which
annihilate the ``Lie type roots''  of ${\frak I}$, so we have the
following definition.

\begin{definition}\label{centerlie}\rm
The {\it Lie-annihilator} of a split Leibniz superalgebra of
maximal length ${\frak L}$ is the set
$$\hbox{${\mathcal Z}_{Lie}({\frak
L}) := \{v \in {\frak L} : [v, {\frak L}_{\alpha}] + [{\frak
L}_{\alpha},v]=0$ for any $\alpha \notin \Lambda^{\frak I}\}$}.$$


\end{definition}

\smallskip

Observe that ${\mathcal Z}({\frak L}) \subset {\mathcal
Z}_{Lie}({\frak L}).$

\begin{proposition}\label{propoI}
Suppose $H = \sum\limits_{\beta \in \Lambda} [{\frak L}_{\beta},
{\frak L}_{-\beta}]$, ${\mathcal Z}_{Lie}({\frak L})=0$ and
${\frak L}$ is root-multiplicative.
 If ${\Lambda}^{\frak I}$ has all of its elements ${\neg {\frak
I}}$-connected and $|\Lambda^{{\frak I}}|>2$, then any nonzero
ideal $I$ of ${\frak L}$ such that $I \subset {\frak I}$ satisfies
$I = {\frak I}$.
\end{proposition}

\begin{proof}
By Equation (\ref{max}) we can write  $$I = (I \cap H) \oplus
(\bigoplus\limits_{\beta \in {\Lambda}^{{I}}_{\bar 0}}{{\frak
L}_{\beta, \bar{0}}} )\oplus (\bigoplus\limits_{\gamma \in
{\Lambda}^{I}_{\bar 1}}{{\frak L}_{\gamma, \bar{1}}} )$$ where
$$\Lambda_{\bar{i}}^{I} := \{\delta \in \Lambda : {
I}_{\bar{i}} \cap {\frak L}_{\delta, \bar{i}} \neq 0\}=\{\delta
\in \Lambda : 0\neq {\frak L}_{\delta, \bar{i}} \subset {
I}_{\bar{i}} \}$$ and with ${\Lambda}^{{I}}_{\bar i} \subset
{\Lambda}^{{\frak I}}_{\bar i}$ for any ${\bar i} \in {\mathbb
Z}_2$.
We begin by showing that
\begin{equation}\label{equhache}
{\frak I} \cap H \subset {\mathcal Z}_{Lie}({\frak L}).
\end{equation}
Indeed, fixed some ${\bar i} \in {\mathbb Z}_2$ and for any
$\alpha \notin \Lambda^{\frak I}$ and $\bar j \in {\mathbb Z}_2$
we have $$[{\frak I} \cap H_{\bar i}, {\frak L}_{\alpha, \bar
j}]+[{\frak L}_{\alpha, \bar j}, {\frak I} \cap H_{\bar i}]
\subset {\frak L}_{\alpha, \bar i + \bar j} \subset {\frak I}.$$
Hence, in case $[{\frak I} \cap H_{\bar i}, {\frak L}_{\alpha,
\bar j}]+[{\frak L}_{\alpha, \bar j}, {\frak I} \cap H_{\bar i}]
\neq 0$ we have  $\alpha \in \Lambda^{\frak I}$, a contradiction.
Hence $[{\frak I} \cap H_{\bar i}, {\frak L}_{\alpha, \bar
j}]+[{\frak L}_{\alpha, \bar j}, {\frak I} \cap H_{\bar i}] = 0$
for each ${\bar i}, {\bar j} \in {\mathbb Z}_2$ and so ${\frak I}
\cap H \subset {\mathcal Z}_{Lie}({\frak L}).$

From the above ${\frak I} \cap H \subset {\mathcal Z}_{Lie}({\frak
L})=0$ and  we can write
\begin{equation}\label{afinal}
{\frak I}=(\bigoplus\limits_{\alpha \in {\Lambda}^{{{\frak
I}}}_{\bar 0}}{{\frak L}_{\alpha, \bar{0}}} )\oplus
(\bigoplus\limits_{\delta \in {\Lambda}^{{\frak I}}_{\bar
1}}{{\frak L}_{\delta, \bar{1}}} ).
\end{equation}

Taking into account ${I}\cap H \subset {\frak I}\cap H=0$, we also
can write
$$I=(\bigoplus\limits_{\beta \in {\Lambda}^{{I}}_{\bar 0}}{{\frak
L}_{\beta, \bar{0}}} )\oplus (\bigoplus\limits_{\gamma \in
{\Lambda}^{I}_{\bar 1}}{{\frak L}_{\gamma, \bar{1}}} )$$ with any
${\Lambda}^{I}_{\bar i} \subset {\Lambda}^{{\frak I}}_{\bar i}$.
Hence, we can take some $\beta_0 \in {\Lambda}^{I}_{\bar i}$ such
that
\begin{equation}\label{betacero}
0 \neq {\frak L}_{\beta_0,{\bar i}} \subset I.
\end{equation}
Now, we can argue with the root-multiplicativity and the maximal
length of ${\frak L}$ as in Proposition \ref{nueva12} to conclude
that given any $\beta \in {\Lambda}^{\frak I}_{\bar j}\setminus
\{\beta_0,-\beta_0\}$, $\bar j \in {\mathbb Z}_2$, there exists a
${\neg {\frak I}}$-connection
\begin{equation}\label{cone}
\{\delta_1,\delta_2,...,\delta_{r-1},\delta_r\}
\end{equation}
from $\beta_0$ to $\beta$ such that
\begin{equation}\label{1003}
0 \neq [[...[{\frak L}_{\beta_0,\bar{i}}, {\frak
L}_{\delta_2,\bar{i}_2}],...],{\frak L}_{\delta_r,\bar{i}_r}] =
{\frak L}_{\beta,\bar{j}} \subset I.
\end{equation}
Now we have to study whether in case $\epsilon \beta_0 \in
\Lambda^{\frak I}_{\bar k}$ for some $\epsilon \in \{1,-1\}$ and
$\bar k \in {\mathbb Z}_2$ then ${\frak L}_{\epsilon \beta_0,
{\bar k}} \subset I$. To do that, observe that the fact
$|\Lambda^{{\frak I}}|>2$ allows us to take an element $\beta \in
\Lambda^{\frak I}_{\bar i}$, for some $\bar i \in {\mathbb Z}_2$,
such that $\beta \notin \{\beta_0,-\beta_0\}$ and, by Equation
(\ref{1003}), satisfying $0 \neq {\frak L}_{\beta,{\bar j}}
\subset I$. Hence we can argue as above with the
root-multiplicativity and maximal length of ${\frak L}$ from
$\beta$ instead of $\beta_0$, to get that in case  $\epsilon
\beta_0 \in \Lambda^{\frak I}_{\bar k}$ for some $\epsilon \in
\{1,-1\}$ and ${\bar k} \in {\mathbb Z}_2$, then $0\neq {\frak
L}_{\epsilon \beta_0, \bar k}\subset I.$

The decomposition of ${\frak I}$ in Equation (\ref{afinal})
finally gives us $I={\frak I}$.
\end{proof}




\begin{theorem}\label{last}
Suppose $$H = \sum\limits_{\beta \in \Lambda} [{\frak L}_{\beta},
{\frak L}_{-\beta}], \hspace{0.1cm} {\mathcal Z_{Lie}}(\frak L)
=0,$$  ${\frak L}$ is root-multiplicative and
$|\Lambda^{\neg{\frak I}}|>2, |\Lambda^{{\frak I}}| >2$. Then
${\frak L}$ is simple if and only if $\Lambda^{\frak I}$ and
$\Lambda^{\neg {\frak I}}$ have (respectively) all of their
elements ${\neg {\frak I}}$-connected.
\end{theorem}

\begin{proof}
Fix some  $\alpha \in \Lambda^{\neg {\frak I}}_{\bar i}$ with
${\bar i} \in {\mathbb Z}_2$, and denote by  $I({\frak
L}_{\alpha,{\bar i}})$ the (graded) ideal of ${\frak L}$ generated
by ${\frak L}_{\alpha,{\bar i}}$. By simplicity $I({\frak
L}_{\alpha,{\bar i}}) = {\frak L}.$ Observe that Remark
\ref{re3.1} together with Super Leibniz identity allow us to
assert that $I({\frak L}_{\alpha,{\bar i}})$ is contained in the
linear span of the set
$$\{[[\cdots[v_{\alpha}, v_{\beta_1}],...],v_{\beta_n}]; \hspace{0.1cm}
[v_{\beta_n},[...[v_{\beta_1},v_{\alpha}],...]];$$
$$\hbox{ $[[...[v_{\beta_1}, v_{\alpha}],...],v_{\beta_n}]; \hspace{0.1cm}
[v_{\beta_n},[...[v_{\alpha}, v_{\beta_1}],...]]$ with $0 \neq
v_{\alpha} \in {\frak L}_{\alpha,{\bar i}}$,}
$$
$$\hbox{ $0 \neq v_{\beta_i} \in {\frak L}_{\beta_i,{\bar j_i}}, \beta_i \in \Lambda$, ${\bar j_i} \in {\mathbb Z}_2$ and $n \in {\hu N}$ \}}$$
being the partial sums $\alpha+\beta_1,
\alpha+\beta_1+\beta_2,..., \alpha+\beta_1+\beta_2+
\cdots+\beta_n$ nonzero. From here, given any $\alpha'\in
\Lambda^{\neg {\frak I}}_{\bar j}$, with ${\bar j}\in {\mathbb
Z}_2$,  the above observation gives us that we can write $\alpha'
= \alpha + \beta_1 + \cdots+ \beta_n$ with any $\beta_i \in
\Lambda^{\neg \frak I}_{\bar j_i}$ and being the partial sums
$\alpha + \beta_1 + \cdots+ \beta_k \in \Lambda^{\neg \frak
I}_{\bar i + \bar j_1 +\cdots + \bar j_k},$ (observe that in case
some $\beta_i \in \Lambda^{\frak I}_{\bar j_i}$ or some partial
sum a root belonging to $\Lambda^{\frak I}_{\bar n}$, then either
the product involving $v_{\beta_i}$ or the partial sum is zero, or
implies $\alpha' \in \Lambda^{\frak I}_{\bar k}$. Hence any
$\beta_i \in \Lambda^{\neg{\frak I}}_{\bar j_i}$ and the partial
sums are in $\Lambda^{\neg{\frak I}}_{\bar n}$). From here, we
have that $\{\alpha, \beta_1,..., \beta_n\}$ is a ${\neg{\frak
I}}$-connection from $\alpha$ to $\alpha'$ and we can assert that
$\hbox{$\Lambda^{\neg{\frak I}}$ has all of its elements ${\neg
{\frak I}}$-connected.}$

If $\Lambda^{\frak I} \neq \emptyset$ and we take some $\beta \in
\Lambda^{\frak I}_{\bar i}$ with ${\bar i} \in {\mathbb Z}_2$, a
similar above argument gives us ${\Lambda}^{\frak I}$ has all of
its elements ${\neg {\frak I}}$-connected.

\smallskip

Let us see the converse. Consider $I$ a nonzero  ideal of ${\frak
L}$ and let us show that necessarily either $I = {\frak I}$ or $I
= {\frak L}$. Let us distinguish two possibilities:
\begin{itemize}
\item If $I \nsubseteq H + {\frak I}$, Proposition \ref{nueva12}
gives us  $I = {\frak L}$.

\item If $I \subset H + {\frak I}$, observe that $I \cap H \subset
{\mathcal Z_{Lie}}(\frak L)$. Indeed, for any $\alpha \notin
\Lambda^{\frak I}$ and $\bar i, \bar j \in {\mathbb Z}_2$ we have
 $[I \cap H_{\bar i}, {\frak L}_{\alpha,\bar j}] +
[{\frak L}_{\alpha,\bar j}, I \cap H_{\bar i}] \subset I \cap
{\frak L}_{\alpha,\bar i + \bar j}  \subset {\frak I}$ and so $[I
\cap H_{\bar i}, {\frak L}_{\alpha,\bar j}] + [{\frak
L}_{\alpha,\bar j}, I \cap H_{\bar i}]=0$. From here
$$I \cap H \subset {\mathcal Z_{Lie}}(\frak L)=0 .$$
Hence  Lemma \ref{lema5} gives us    $I \subset {\frak I}$.
Finally  Proposition \ref{propoI} implies $I = {\frak I}$ and the
proof is complete.
\end{itemize}
\end{proof}

It remains to study the cases in which either $|\Lambda^{\neg
{\frak I}}| \leq 2 $ or $|\Lambda^{{\frak I}}| \leq 2.$

\smallskip

\begin{proposition}\label{cardinal2}
Suppose    $$H = \sum\limits_{\beta \in \Lambda} [{\frak
L}_{\beta}, {\frak L}_{-\beta}], \hspace{0.1cm} {\mathcal
Z_{Lie}}(\frak L) =0,$$  ${\frak L}$ is root-multiplicative and
$\Lambda^{\neg {\frak I}}$, $\Lambda^{\frak I}$  have
(respectively) all of their elements ${\neg {\frak I}}$-connected.
If either $|\Lambda^{\neg{\frak I}}| \leq 2$ or $|\Lambda^{{\frak
I}}| \leq 2$, then one of the following assertions holds.
\begin{enumerate}
\item[{\rm (1)}] ${\frak L}$ is a simple split  Leibniz
superalgebra.

\item[{\rm (2)}]
${\rm char}({\mathbb K}) \neq 2$ and ${\frak L} = H \oplus
(\bigoplus\limits_{\alpha \in \Lambda_{\bar{0}}^{\neg {\frak I}}}
{\frak L}_{\alpha, \bar{0}}) \oplus (\bigoplus\limits_{\beta \in
\Lambda_{\bar{1}}^{\neg {\frak I}}} {\frak L}_{\beta, \bar{1}})
\oplus I\oplus K$ with ${\frak I}=I \oplus K$ and where either
\begin{itemize}
\item[{\rm (i)}]  $I={\frak L}_{\gamma, \bar{i}}\oplus {\frak
L}_{-\gamma, \bar{j}}$ is a two-dimensional ideal of ${\frak L}$
and $K={\frak L}_{\gamma, \bar{i} + \bar{1}} \oplus {\frak
L}_{-\gamma, \bar{j}+ \bar{1}}$ is a one or two-dimensional
(abelian) subalgebra of ${\frak L}$ (depending on either ${\frak
L}_{\gamma, \bar{i}+ \bar{1}}=0$ or ${\frak L}_{-\gamma, \bar{j}+
\bar{1}}=0$), for ${\bar i}, {\bar j} \in {\mathbb Z}_2$; or

\item[{\rm (ii)}] $I={\frak L}_{\gamma, \bar{0}}\oplus {\frak
L}_{\gamma, \bar{1}}\oplus {\frak L}_{-\gamma, \bar{i}}$ is a
three-dimensional ideal of ${\frak L}$ and $K={\frak L}_{-\gamma,
\bar{i}+ \bar{1}}$ a one-dimensional (abelian) subalgebra of
${\frak L}$, for  $ {\bar i} \in {\mathbb Z}_2$.

\end{itemize}

\item[{\rm (3)}] ${\rm char}({\mathbb K}) = 2$ and ${\frak L} = I
\oplus K \oplus (\bigoplus\limits_{\gamma \in
\Lambda_{\bar{0}}^{{\frak I}}\setminus \Lambda_{\bar{0}}^{I}}
{\frak L}_{\gamma, \bar{0}}) \oplus (\bigoplus\limits_{\delta \in
\Lambda_{\bar{1}}^{ {\frak I}}\setminus \Lambda_{\bar{1}}^{I}}
{\frak L}_{\delta, \bar{1}})$ where $I= ([{\frak L}_{\alpha,
\bar{i}}, {\frak L}_{\alpha, \bar{i}+\bar{1}}]+[{\frak
L}_{\alpha,\bar{i}+ \bar{1}},{\frak L}_{\alpha, \bar{i}}]) \oplus
{\frak L}_{\alpha, \bar{i}}\oplus (\bigoplus\limits_{\tau \in
\Lambda_{\bar{0}}^{{\frak I}}\cap \Lambda_{\bar{0}}^I} {\frak
L}_{\tau, \bar{0}}) \oplus (\bigoplus\limits_{\eta \in
\Lambda_{\bar{1}}^{ {\frak I}}\cap \Lambda_{\bar{1}}^I} {\frak
L}_{\eta, \bar{1}})$ is a nonzero ideal of
${\frak L}$, $I \neq {\frak I}$; and $K=[{\frak
L}_{\alpha,\bar{i}+ \bar{1}}, {\frak L}_{\alpha, \bar{i}+\bar{1}}]
\oplus {\frak L}_{\alpha, \bar{i}+\bar{1}}$ is a two-dimensional
subalgebra of ${\frak L}$, for ${\bar i} \in {\mathbb Z}_2$.

\item[{\rm (4)}] ${\rm char}({\mathbb K}) \neq 2$ and ${\frak L} =
I \oplus K $ where either
\begin{itemize}
  \item[{\rm (i)}] $I= H_{\bar 1} \oplus {\frak L}_{\alpha, \bar{i}} \oplus  (\bigoplus\limits_{\tau \in
\Lambda_{\bar{0}}^{{\frak I}}\cap \Lambda_{\bar{0}}^{I}} {\frak
L}_{\tau, \bar{0}}) \oplus (\bigoplus\limits_{\eta \in
\Lambda_{\bar{1}}^{ {\frak I}}\cap \Lambda_{\bar{1}}^{I}} {\frak
L}_{\eta, \bar{1}})$ is an  ideal
 of ${\frak L}$
and $K=H_{\bar 0} \oplus {\frak L}_{\alpha,\bar{i}+ \bar{1}}
\oplus {\frak L}_{-\alpha,\bar{i}+ \bar{1}}\oplus
(\bigoplus\limits_{\gamma \in \Lambda_{\bar{0}}^{{\frak
I}}\setminus \Lambda_{\bar{0}}^{I}} {\frak L}_{\gamma, \bar{0}})
\oplus (\bigoplus\limits_{\delta \in \Lambda_{\bar{1}}^{ {\frak
I}}\setminus \Lambda_{\bar{1}}^{I}} {\frak L}_{\delta, \bar{1}})$
with ${\rm dim}({\frak L}_{\alpha, \bar{i}} \oplus  {\frak
L}_{\alpha,\bar{i}+ \bar{1}} \oplus {\frak L}_{-\alpha,\bar{i}+
\bar{1}})=3$,
 for ${\bar i} \in {\mathbb Z}_2$, or

\item[{\rm (ii)}] $I= (I\cap H_{\bar 0}) \oplus H_{\bar 1} \oplus
{\frak L}_{\alpha, \bar{i}} \oplus {\frak L}_{-\alpha, \bar{i}}
\oplus (\bigoplus\limits_{\tau \in \Lambda_{\bar{0}}^{{\frak
I}}\cap \Lambda_{\bar{0}}^{I}} {\frak L}_{\tau, \bar{0}}) \oplus
(\bigoplus\limits_{\eta \in \Lambda_{\bar{1}}^{ {\frak I}}\cap
\Lambda_{\bar{1}}^{I}} {\frak L}_{\eta, \bar{1}})$ is an ideal
 of ${\frak L}$
and $$K=K_{\bar 0} \oplus {\frak L}_{\alpha,\bar{i}+ \bar{1}}
\oplus {\frak L}_{-\alpha,\bar{i}+ \bar{1}} \oplus
(\bigoplus\limits_{\gamma \in \Lambda_{\bar{0}}^{{\frak
I}}\setminus \Lambda_{\bar{0}}^{I}} {\frak L}_{\gamma, \bar{0}})
\oplus (\bigoplus\limits_{\delta \in \Lambda_{\bar{1}}^{ {\frak
I}}\setminus \Lambda_{\bar{1}}^{I}} {\frak L}_{\delta, \bar{1}}),
$$ where $$K_{\bar 0}=\sum\limits_{\tiny{[{\frak L}_{\epsilon
\alpha, \bar{i} + \bar{1}},{\frak L}_{-\epsilon \alpha, \bar{i} +
\bar{1}}] \cap I=0; \epsilon \in \{1,-1\}}} [{\frak L}_{\epsilon
\alpha, \bar{1} + \bar{1}},{\frak L}_{-\epsilon \alpha, \bar{1} +
\bar{1}}]$$ and ${\rm dim}({\frak L}_{\alpha, \bar{i}} \oplus
{\frak L}_{-\alpha, \bar{i}} \oplus {\frak L}_{\alpha,\bar{i}+
\bar{1}} \oplus {\frak L}_{-\alpha,\bar{i}+ \bar{1}})=4$,  for
${\bar i} \in {\mathbb Z}_2$, or

\item[{\rm (iii)}] $I= (I\cap H_{\bar 0}) \oplus (I \cap H_{\bar
1} ) \oplus {\frak L}_{\alpha, \bar{i}}  \oplus
(\bigoplus\limits_{\tau \in \Lambda_{\bar{0}}^{{\frak I}}\cap
\Lambda_{\bar{0}}^{I}} {\frak L}_{\tau, \bar{0}}) \oplus
(\bigoplus\limits_{\eta \in \Lambda_{\bar{1}}^{ {\frak I}}\cap
\Lambda_{\bar{1}}^{I}} {\frak L}_{\eta, \bar{1}})$ is an ideal
 of ${\frak L}$
and $$K=K_{\bar 0} \oplus K_{\bar 1} \oplus   {\frak
L}_{-\alpha,\bar{i}} \oplus {\frak L}_{\alpha,\bar{i}+ \bar{1}}
\oplus {\frak L}_{-\alpha,\bar{i}+ \bar{1}} \oplus
(\bigoplus\limits_{\gamma \in \Lambda_{\bar{0}}^{{\frak
I}}\setminus \Lambda_{\bar{0}}^I} {\frak L}_{\gamma, \bar{0}})
\oplus (\bigoplus\limits_{\delta \in \Lambda_{\bar{1}}^{ {\frak
I}}\setminus \Lambda_{\bar{1}}^I} {\frak L}_{\delta, \bar{1}})$$
where  any $K_{\bar j}=\sum\limits_{\tiny{[{\frak L}_{\epsilon
\alpha, \bar{k}},{\frak L}_{-\epsilon \alpha, \bar{k}+\bar{j}}]
\cap I=0; \epsilon \in \{1,-1\}}; {\bar k} \in {\mathbb Z}_2}
[{\frak L}_{\epsilon \alpha, \bar{k} },{\frak L}_{-\epsilon
\alpha, \bar{k}+\bar{j} }]$  and ${\rm dim}({\frak L}_{\alpha,
\bar{i}} \oplus {\frak L}_{-\alpha, \bar{i}} \oplus {\frak
L}_{\alpha,\bar{i}+ \bar{1}} \oplus {\frak L}_{-\alpha,\bar{i}+
\bar{1}})=4$, for ${\bar i} \in {\mathbb Z}_2$.
 \end{itemize}
\end{enumerate}
\end{proposition}

\begin{proof}
Now, suppose ${\frak L}$ is not  simple, then there exists a
nonzero ideal $I$ of ${\frak L}$ such that $I \notin \{{\frak I},
{\frak L} \}$. Let us distinguish three cases:

\smallskip

1. $|\Lambda^{{\frak I}}| \leq 2$ and $|\Lambda^{\neg {\frak
I}}|>2.$ If $I \subsetneq H+ {\frak I}$, we have as in Proposition
\ref{nueva12} that $I={\frak L}$. From here $I \subset H+ {\frak
I}.$ Since we get $I\cap H=0$ as in Theorem \ref{last} we have
$$I \subset {\frak I}.$$
Let us show  that necessarily
$$|\Lambda^{{\frak I}}| \neq 1.$$
Indeed, if $|\Lambda^{{\frak I}}| =1$, then $|\Lambda^{{\frak I}}|
=\{\gamma\}$ and so either $\gamma \in \Lambda^{\frak I}_{\bar 0}
\cup \Lambda^{\frak I}_{\bar 1}$ or $\gamma \in \Lambda^{\frak
I}_{\bar i}$ with $\gamma \notin \Lambda^{\frak I}_{\bar i+1}.$
We have as in Proposition \ref{propoI} that some $0 \neq {\frak
L}_{\gamma, \bar{j}} \subset I$, (in the second case $\bar{j}
=\bar{i}$). Since ${\mathcal Z}_{Lie}({\frak L})=0$, there exists
$\alpha \in \Lambda_{\bar{k}}^{\neg {\frak I}}$ such that $
[{\frak L}_{\gamma, \bar{j}}, {\frak L}_{\alpha, \bar{k}}]+ [
{\frak L}_{\alpha, \bar{k}},{\frak L}_{\gamma, \bar{j}}] \neq 0$
and so $\gamma+\alpha \in \Lambda_{\bar{j}+ \bar{k}}^{ {\frak
I}}.$ From here $\gamma + \alpha = \gamma$ and so $\alpha=0$, a
contradiction. Hence the possibility $|\Lambda^{{\frak I}}| =1$
never happens. From here, we just have to consider the case in
which $|\Lambda^{{\frak I}}| = 2.$
Then either $\Lambda^{{\frak
I}}=\{ \gamma, \delta\}$ with $\delta \neq -\gamma$ or
$\Lambda^{{\frak I}}=\{\gamma,-\gamma\}$ and ${\rm char }({\mathbb
K}) \neq 2$. In the first possibility we can argue as in
Proposition \ref{propoI} to get $I={\frak I}$ and so we just have
to center on the second one, that is, $\Lambda^{{\frak
I}}=\{\gamma,-\gamma\}$ and ${\rm char }({\mathbb K}) \neq 2$. As
in the case in which $|\Lambda^{{\frak I}}| =1$, we have that some
$0 \neq {\frak L}_{\epsilon \gamma, \bar{j}} \subset I$ with
$\epsilon \in \{1,-1\}$ and that there exists $\alpha \in
\Lambda_{\bar{k}}^{\neg {\frak I}}$ such that $ [{\frak
L}_{\epsilon \gamma, \bar{j}}, {\frak L}_{\alpha, \bar{k}}]+
[{\frak L}_{\alpha, \bar{k}},{\frak L}_{\epsilon \gamma, \bar{j}}]
\neq 0$ being then $\epsilon\gamma+\alpha \in \Lambda_{\bar{j}+
\bar{k}}^{ {\frak I}}.$ From here, either $\epsilon \gamma +
\alpha = \epsilon\gamma$ or $\epsilon \gamma + \alpha =
-\epsilon\gamma.$ Since the first possibility never happens we get
$\alpha=-2\epsilon \gamma$ and so $0\neq [{\frak L}_{\epsilon
\gamma, \bar{j}}, {\frak L}_{\alpha, \bar{k}}]+ [{\frak
L}_{\alpha, \bar{k}},{\frak L}_{\epsilon \gamma, \bar{j}}]={\frak
L}_{-\epsilon \gamma, \bar{j}+\bar{k}} \subset I$. Hence, if
$\epsilon \gamma \notin \Lambda^{ \frak I}_{\bar j + \bar 1}$ and
$-\epsilon \gamma \notin \Lambda^{ \frak I}_{\bar{j}+\bar{k} +
\bar{1}}$ we get $I={\frak I}$. From here, we just have to study
the opposite case in which it is straightforward to verify that we
necessarily  have  the option (2) of the Proposition.

\smallskip

2. $|\Lambda^{\neg {\frak I}}| \leq 2$ and $|\Lambda^{{\frak
I}}|>2.$ If $I \subset H+ {\frak I}$, we have as in item 1. that
$I \subset {\frak I}$ and then $I={\frak I}$ by Proposition
\ref{propoI}. From here, $$I \subsetneq H+ {\frak I}.$$ Since
 Equation (\ref{equsimple}) gives us
$|\Lambda^{\neg {\frak I}}| \neq 0$, (in the opposite case $H=0$),
we get $|\Lambda^{\neg {\frak I}}| \in \{1,2\}$ and so we can
distinguish two possibilities:
\begin{itemize}
\item If $|\Lambda^{\neg {\frak I}}| =1$, Equation
(\ref{equsimple})  implies that necessarily ${\rm char}({\mathbb
K}) =2$. By taking also into account  Lemma \ref{lema5}, the fact
$\alpha(H_{\bar 0}) \neq 0$ for any nonzero root $\alpha$; and
that ${\mathcal Z_{Lie}}(\frak L) =0$ we get in a straightforward
way the case (3) of the Proposition.

\item If $|\Lambda^{\neg {\frak I}}| =2$,  then either
$\Lambda^{\neg {\frak I}}=\{ \alpha, \beta\}$ with $\beta \neq
-\alpha$ or $\Lambda^{\neg{\frak I}}=\{\alpha,-\alpha\}$ and ${\rm
char }({\mathbb K}) \neq 2$. In the first possibility we have as
in Proposition \ref{nueva12} that $I={\frak L}$ and so we just
have to consider the second one, that is, $\Lambda^{\neg {\frak
I}}=\{\alpha,-\alpha\}$ and ${\rm char}({\mathbb K}) \neq 2$. By a
routinary case by case study on ${\rm
dim}(\bigoplus\limits_{\gamma \in \Lambda^{\neg{\frak I}}} {\frak
L}_{\gamma})\in \{2,3,4\}$, taking into account Equation
(\ref{equsimple}), Lemma \ref{lema5}, the fact $\alpha(H_{\bar 0})
\neq 0$ for any nonzero root $\alpha$ and that ${\mathcal
Z_{Lie}}(\frak L) =0$, we obtain the possibility  (4) of the
Proposition.
\end{itemize}

\smallskip

3.  $|\Lambda^{{\frak I}}| \leq 2$ and $|\Lambda^{\neg {\frak I}}|
\leq 2.$ If $I \subset H + {\frak I}$ we have as in item 1. that
$I \subset {\frak I}$ and $|\Lambda^{\frak I}| \neq 1$, being then
$|\Lambda^{\frak I}| =2$. By arguing as in item 1. we get that
either $I={\frak I}$ or we are in the case (2) of the Proposition.

If $I \nsubseteq H + {\frak I}$, by arguing as in item 2. we
obtain that necessarily  either possibility (3) or possiblity (4)
of the Proposition holds.
\end{proof}

\end{document}